\documentclass[12pt]{amsart}
\usepackage{amsmath}
\usepackage{amsthm}
\usepackage{amssymb}
\usepackage{amsfonts,mathrsfs}
\usepackage{stmaryrd}
\usepackage{amsxtra}
\usepackage{epsfig}
\usepackage{verbatim}

\theoremstyle{plain}
\newtheorem{theorem}[equation]{Theorem}

\newtheorem{lemma}[equation]{Lemma}

\newtheorem{definition}[equation]{Definition}
\theoremstyle{remark}

\numberwithin{equation}{section}

\include{header}

\newcommand{\fa}{\mathfrak{a}}

\newcommand{\fb}{\mathfrak{b}}

\newcommand{\closedfa}{\overline{\mathfrak{a}}}

\newcommand{\Jfca}{\mathcal{J}(c\cdot\mathfrak{a})}
\newcommand{\JfcaAN}{\mathcal{J}_{an}(c\cdot\mathfrak{a})}
\newcommand{\JfaAN}{\mathcal{J}_{an}(\mathfrak{a})}
\newcommand{\JfcaAL}{\mathcal{J}_{al}(c\cdot\mathfrak{a})}
\newcommand{\JfaAL}{\mathcal{J}_{al}(\mathfrak{a})}

\newcommand{\closedfaA}{\overline{\mathfrak{a}^{A}}^{an}}
\newcommand{\closedfaAN}{\overline{\mathfrak{a}}^{an}}
\newcommand{\closedfbAN}{\overline{\mathfrak{b}}^{an}}
\newcommand{\closedfaAL}{\overline{\mathfrak{a}}^{al}}
\newcommand{\closedfaAminus}{\overline{\mathfrak{a}^{A-1}}^{an}}
\newcommand{\faA}{\mathfrak{a}^{A}}
\newcommand{\faAminus}{\mathfrak{a}^{A-1}}

\newcommand{\one}{\vec{\mathbf{1}}}

\begin{document}

\bibliographystyle{plain}
\title[Multiplier Ideals and Integral Closure of Monomial Ideals]{Multiplier Ideals and Integral Closure of Monomial Ideals: An Analytic Approach}
\author{Jeffery D. McNeal \& Yunus E. Zeytuncu}
\thanks{Research  of both authors was partially supported by NSF grants}
\address{Department of Mathematics, \newline Ohio State University, Columbus, Ohio 43210}
\email{mcneal@math.ohio-state.edu}
\email{yunus@math.ohio-state.edu}
\begin{abstract}
Proofs of two results about a monomial ideal -- describing membership in auxiliary ideals associated to the monomial ideal -- are given which do not invoke resolution of singularities. The AM--GM inequality is used as a substitute for taking a log resolution of the monomial ideal.
\end{abstract}
\maketitle

\section{Introduction}

Multiplier ideals\footnote{There is not a unique notion of ``multiplier ideals'': (i) sets of multipliers can be associated to a given geometric object via different boundedness conditions, and many of these sets form ideals, (ii) multipliers can be associated to many different objects; e.g., plurisubharmonic functions, divisors in projective manifolds, real hypersurfaces in $\mathbb{C}^n$, metrics on line bundles, have naturally associated but somewhat different multiplier ideals, and (iii) {\it analytic} multipliers are naturally defined using transcendental data (holomorphic functions, real coefficients, etc.) while {\it algebraic} multipliers involve only algebraic data (polynomial functions, rational coefficients, etc.).

We side-step all these issues here, by dealing with the notion of multiplier ideal of greatest current interest and by restricting it to the polynomial ring.} have a crucial role in some significant recent work in algebraic and analytic geometry. For example, the prominent work in \cite{DemaillyKollar}, \cite{Nadel90}, \cite{Siu93}, \cite{Siu96}, \cite{Siu98}, \cite{Siu00} use these ideals as a fundamental tool and these ideals underly the recent progress on the Minimal Model program described in \cite{demaillypcmi}. The utility of multiplier ideals, in these works and other applications, stem from the dual properties with respect to vanishing theorems and measurement of singularities that these ideals exhibit.

There are two ways to define the multiplier ideals commonly appearing in the literature --- analytically and algebraically --- and the two definitions look initially dissimilar.
This is true even when considering ideals (both the input ideal and its multiplier ideal) in the polynomial ring $\mathbb{C}\left[z_1,\dots,z_n\right]$, where the definitions give rise to the same multiplier ideal. Given a non-zero ideal $\fa=\left(p_1(z),\dots,p_k(z)\right)\subset \mathbb{C}\left[z_1,\dots,z_n\right]$  generated by $k$ polynomials, define another ideal associated to $\fa$ by

\begin{equation}\label{an_multiplier}
\JfaAN =\left\{g\in \mathbb{C} \left[z_1,\dots,z_n\right]  : \hskip .2 cm \frac{|g(z)|^2}{|p_1(z)|^{2}+\dots+|p_k(z)|^{2}} \in L^1_{loc}(\mathbb{C}^n)\right\}.
\end{equation}
Following Lazarsfeld \cite{Lazarsfeldbook}, call this the  \textit{analytic multiplier ideal} associated to $\fa$. On the other hand, for the algebraic definition, first fix a \textit{log resolution} of $\fa$: $\mu:X'\to X$, with $\fa\cdot \mathcal{O}_{X'}=\mathcal{O}_{X'}(-F)$. Then the ideal defined as

\begin{equation}\label{al_multiplier}
\JfaAL=\mu_{\ast}\mathcal{O}_{X'}\left(K_{X'/ X}-[F]\right)
\end{equation}
is called the \textit{algebraic multiplier ideal} associated to $\fa$.\footnote{The notation is explained in the next section, where definitions  \eqref{an_multiplier} and \eqref{al_multiplier} are also generalized.}

Using either formulation \eqref{an_multiplier} or \eqref{al_multiplier}, it is difficult to actually compute the multiplier ideal associated to a given ideal $\fa$. A large class of ideals $\fa$ for which \eqref{al_multiplier} has been computed, however, are those generated by monomials. In \cite{Howald2001}, Howald showed that if $\fa\subset \mathbb{C}  \left[z_1,\dots,z_n \right]   $ is generated by monomials, then $\JfaAL$ can be described in terms of the Newton polyhedra associated to the exponent set of the generators (see Theorem \ref{main} in section 3). To establish this result, Howald must, per definition, take a log resolution of $\fa$, which he then analyzes using some facts about toric varieties.

\vskip .5cm
In the third section of the paper, we give an elementary analytic proof of Howald's theorem. This proof uses only the simple
Arithmetic Mean-Geometric Mean inequality
\begin{equation}\label{AGinequality}
x_1^{t_1}\dots x_k^{t_k}\leq t_1x_1+\dots+t_kx_k\hskip .5 cm\left(\leq x_1+\dots +x_k\right)
\end{equation}
for any $x_i\geq 0$ and $t_i\geq 0$ such that $t_1+\dots+t_k=1$, and a careful examination of the Newton polyhedra associated to the monomial ideal $\fa$. The use of \eqref{AGinequality} as a substitute for log resolving $\fa$ is the main point of this paper.

Indeed, replacing the resolution of singularities arguments in \cite{Howald2001} by the AM-GM inequality was, for the authors, a significant conceptual simplification.
Since monomial ideals are interesting in their own right, and also model general ideals under certain circumstances, for example see \cite{deFernex} and \cite{MustataGraded}, our alternate proof of Howald's basic result may be of some interest to others.
\vskip .5cm

In the fourth section of the paper, similar elementary ideas are applied to the notion of integral closure. The AM-GM inequality is used to describe the integral closure of a monomial ideal via the Newton polyhedra and then this description is used to prove the equivalence of two definitions of integral closure for monomial ideals. This result is not new:
the equivalence of the two notions of integral closure, for an arbitrary polynomial ideal, was established by Teissier, see \cite{teissier82} and \cite{Teissier2008}. But the method of proof in Section 4, which again hinges on using \eqref{AGinequality} rather than log resolution, is the reason for its inclusion. 
\vskip .5cm

For the basic facts, examples, applications, and alternate definitions of multiplier ideals, we refer the reader to several well-written sources. For information on the algebraic aspects of multiplier ideals, we recommend both Lazarsfeld's book \cite{Lazarsfeldbook}, Chapter 9, and the survey article of Blickle-Lazarsfeld \cite{BlickleLazarsfeld}. For information on the analytic aspects of multiplier ideals, we recommend the two sets of lecture notes by Demailly, \cite{demaillytrieste}, \cite{demaillypcmi}.\\

This project grew out of the Park City Mathematical Institute's program during the summer of 2008. In particular, the lectures of Robert Lazarsfeld at PCMI2008 and the year-long AAG meetings at Ohio State, that were designed to follow-up and expand on the material discussed in PCMI2008, inspired us to write this paper.
We thank Rob Lazarsfeld, Dror Varolin, Mircea Musta\c{t}\u{a}, and the main participants in the AAG meetings --- Herb Clemens, Gary Kennedy, Mirel Caibar, Jie Wang, Yu-Han Liu, Janhavi Joshi, Sivaguru Ravisankar, and Wing-San Hui --- for their stimulating conversations regarding this material.

\section{Definitions and notation}

We start by folding a rational parameter, $c>0$, into the definition of $\JfaAN$.

\begin{definition}\label{an_cmultiplier} If  $\fa=(p_1(z),\dots,p_k(z))\subset \mathbb{C}  \left[z_1,\dots,z_n \right]   $  is a non-zero ideal generated by $k$ polynomials and $c\in\mathbb{Q}^+$,
the ideal corresponding to $\fa$ defined by

\begin{equation}\label{an_cmultiplier_eq}
\JfcaAN =\left\{g\in \mathbb{C}  \left[z_1,\dots,z_n \right]    : \hskip .2 cm \frac{|g(z)|^2}{\sum_{j=1}^k |p_j(z)|^{2c}} \in L^1_{loc}(\mathbb{C}^n)\right\},
\end{equation}
is called the analytic multiplier ideal of $\fa$ of depth $c$.
\end{definition}

In Definition \ref{an_cmultiplier}, Lebesgue measure, $dm$, is tacitly used, i.e. $f\in L^1_{loc}(\mathbb{C}^n)$ if $\int_K |f| \, dm < \infty$ for all compact $K\subset \mathbb{C}^n$. Other measures could be used instead. There is no meaning assigned to $c\cdot\fa$ in \eqref{an_cmultiplier_eq}, independent of its role in indexing $\JfcaAN$; in other articles, e.g. \cite{Lazarsfeldbook}, $\JfcaAN$ is written multiplicatively as $\mathcal{J}_{an}(\mathfrak{a}^c)$.

It follows directly from Definition \ref{an_cmultiplier} that the ideals $\JfcaAN$ measure the order of vanishing of the zero-variety of $\fa$,
\begin{equation}\label{zero_locus}
V(\fa)=\left\{z\in\mathbb{C}^n:p_1(z)=\dots=p_k(z)=0\right\}.
\end{equation}
For instance, if $c>0$ is small enough (depending on the degrees of the polynomials $p_1,\dots, p_k$), then $\JfcaAN$ is the full polynomial ring $ \mathbb{C}  \left[z_1,\dots,z_n \right]   $. And in the opposite direction, if $x_0\in V(\fa)$ and all the generators of $\fa$ vanish to high order at $x_0$, then the integrability condition for $g\in\mathcal{J}_{an}(1\cdot\mathfrak{a})$ forces $g$ to vanish appropriately at $x_0$. Similarly, for any ideal $\fa$, $\JfcaAN$ becomes farther and farther from the full ring $ \mathbb{C}  \left[z_1,\dots,z_n \right]   $ as $c\to+\infty$.

In order to define $\JfcaAL$, we recall two standard notions in algebraic geometry. A \textit{log resolution} of the ideal $\fa$ is a proper, birational map $\mu: X'\longrightarrow\mathbb{C}^n$ whose exceptional locus is a divisor, $E_\mu$, and
\begin{enumerate}
\item $X'$ is non-singular.
\item  For some effective divisor $F=\sum a_j E_j$,
$$ \fa\cdot\mathcal{O}_{X'} =\mu^{-1}(\fa) =\mathcal{O}_{X'}\left(-F\right). $$
\item The divisor $F+E_\mu$ has simple normal-crossing support.
\end{enumerate}
A log resolution exists for any ideal in $\mathbb{C}  \left[z_1,\dots,z_n \right]$, a fact that follows from the fundamental work of Hironaka, \cite{hironaka}.

The second notion is the \textit{relative canonical bundle}: given $\mu: X'\longrightarrow\mathbb{C}^n$ a proper, birational map as above, the relative canonical bundle associated to this map is

\begin{equation}
	K_{X'/ \mathbb{C}^n} =: K_{X'} - \mu^*\left(K_{\mathbb{C}^n}\right).
\end{equation}
In divisor notation, this bundle is expressed as $\text{Div}\left(\det\left(\text{Jac}_{\mathbb{C}^n}\mu\right)\right)$.
For information on the terms appearing in these two definitions, see \cite{Lazarsfeldbook} and \cite{eisenbudbook}. 
\begin{definition}\label{al_cmultiplier} Let  $\fa=(p_1(z),\dots,p_k(z))\subset \mathbb{C}  \left[z_1,\dots,z_n \right]   $  be a non-zero ideal generated by $k$ polynomials and $c\in\mathbb{Q}^+$.
Let $\mu: X'\longrightarrow\mathbb{C}^n$ be a log resolution of $\fa$ and $F$ the effective divisor such that $\fa\cdot\mathcal{O}_{X'}= \mathcal{O}_{X'}\left(-F\right) $.
Then the algebraic multiplier ideal of $\fa$ of depth $c$ is defined to be

\begin{equation}\label{al_multiplier_eq}
\JfcaAL=\mu_{\ast}\mathcal{O}_{X'}\left(K_{X'/ \mathbb{C}^n}-\lfloor cF\rfloor\right),
\end{equation}
where $\lfloor cF\rfloor$ denotes the round-down of the divisor $cF$.
\end{definition}
The definition of $\JfcaAL$ does not depend on the particular log resolution of $\fa$, see \cite{Lazarsfeldbook}, pages 156--158. For basic properties of $\JfcaAL$, like subadditivity, see\cite{DEL} and \cite{MustataSum}.

An initial reason for interest in $\JfcaAN$ and $\JfcaAL$ is obvious: both measure how complicated \eqref{zero_locus} is. However, the special interest in
$\JfaAN$ and $\JfaAL$ arises from the remarkable properties, both local and global, that these particular ideals possess and because of the powerful cohomological vanishing statements that can be inferred using them; cf. \cite{Lazarsfeldbook}, Chapter 9,  \cite{BlickleLazarsfeld}, \cite{demaillytrieste}, and \cite{demaillypcmi}. From the analytic side, the special properties of $\JfaAN$ stem from the fact that the membership condition \eqref{an_cmultiplier_eq} involves the $L^2$ norm, rather than another boundedness condition.

\vskip .25 cm

A monomial ideal is an ideal $\fa\subset \mathbb{C}  \left[z_1,\dots,z_n \right]   $ that is generated by monomials. The notation $\fa=(z^{\alpha^1},\dots,z^{\alpha^k})$ will denote the ideal generated by $k$ monomials where each $\alpha^i=(\alpha^i_1,\dots,\alpha^i_n)\in \left(\mathbb{N}\cup \left\{0\right\}\right)^n$ is a multi-index.

Any monomial $z^{\alpha}=z_1^{\alpha_1}\dots z_n^{\alpha_n}\in \mathbb{C}  \left[z_1,\dots,z_n \right]   $ corresponds to a point of the lattice $L=\mathbb{N}\cup\{0\}\times\dots\times\mathbb{N}\cup\{0\}\subset \left(\mathbb{R}_{+}\cup\{0\}\right)^n$ given by $\left(\alpha_1,\dots,\alpha_n\right)$.
The set of generators $\{z^{\alpha^1},\dots,z^{\alpha^k}\}$ of $\fa$ thus determines a certain subset of $L$. Let $C$ denote the convex hull of this set in $\left(\mathbb{R}_{+}\cup\{0\}\right)^n$.
\begin{definition}\label{Newton}
The (solid) Newton polyhedra of the ideal $\fa$ is defined as
\begin{equation}
P(\fa)=\bigcup_{\mathbf{x}\in C} \mathbf{x}+\left(\mathbb{R}_{+}\cup\{0\}\right)^n.
\end{equation}
If $c\in\mathbb{Q}^+$, the scaled Newton polyhedra of depth $c$ is as
\begin{equation}
P\left(c\fa\right)=\left\{ c\cdot\mathbf{x}: \mathbf{x}\in P(\fa)\right\}.
\end{equation}
\end{definition}

For the rest of the note, the following convenient, though non-standard, notation will be used. For two multi-indices $\alpha,\beta$,  write $\alpha \prec \beta$ to denote that $\alpha_i<\beta_i$ for all $i=1,\dots,n$ and $\alpha \preceq \beta$ for $\alpha_i\leq \beta_i$ for all $i=1,\dots,n$. Also, the expression $A\lesssim B$ will mean that there exists a constant $k>0$ such that $A\leq kB$. Finally, the symbol $\one$ denotes the n-tuple $(1,\dots,1)$ and $P^{\circ}(\fa)$ denotes the interior of $P(\fa)$.\\

\section{Howald's theorem}

Howald's description of the multiplier ideal associated to a monomial ideal is

\begin{theorem}\label{main}
Let $\fa$ be a monomial ideal, $\fa=(z^{\alpha^1},z^{\alpha^2},\dots,z^{\alpha^k})$, then
\begin{enumerate}
\item[(i)] $\JfcaAL$ is also a monomial ideal.
\item[(ii)] $z^{\beta}\in \JfcaAL$ if and only if $\beta+\one \in P^{\circ}(c\fa)$.
\end{enumerate}
\end{theorem}

We now give a proof of Theorem \ref{main} using the definition of $\JfcaAN$ instead of $\JfcaAL$.

\begin{proof}

The equivalence in (ii) will be proved first. Let $\beta$ be a multi-index such that $\beta+\one\in P^{\circ}(c\fa)$. From Definition \ref{Newton}, it is simple to observe that there exists numbers $t_i\geq0$ satisfying $t_1+\dots+t_k=1$ such that $t_1c\alpha^1+\dots+t_kc\alpha^k\prec \beta+\one$.

The first step is to check the integrability of
\begin{equation}\label{ratio}
\frac{|z^{\beta}|^2}{|z^{\alpha^1}|^{2c}+\dots+|z^{\alpha^k}|^{2c}}\hskip 1cm\text{near the origin.}
\end{equation}
Let $U$ be the unit polydisc centered at the origin in $\mathbb{C}^n$. Applying \eqref{AGinequality} to the denominator in \eqref{ratio} yields
\begin{align*}
\int_U \frac{|z^{\beta}|^2}{|z^{\alpha^1}|^{2c}+\dots+|z^{\alpha^k}|^{2c}}\, dV(z)\leq \int_U \frac{|z^{\beta}|}{|z^{\alpha^1}|^{2ct_1}\dots |z^{\alpha^k}|^{2ct_k}}\, dV(z)&\\
\lesssim \int_{[0,1]^n}\frac{r_1^{2\beta_1+1}\dots r_n^{2\beta_n+1}}{r_1^{2c(\alpha^1_1t_1+\dots+\alpha^k_1t_k)}\dots r_n^{2c(\alpha^1_nt_1+\dots+\alpha^k_nt_k)}}\,\, dr_1\dots dr_n&\\
\lesssim \prod_{i=1}^n\int_0^1  \frac{r_i^{2(\beta_i+1-c(\alpha^1_it_1+\dots +\alpha^k_it_k))}}  {r_i}\, dr_i&.
\end{align*}
But $t_1c\alpha^1+\dots +t_kc\alpha^k\prec \beta+\one$ implies that the each of the one variable integrals above is finite and therefore $\frac{|z^{\beta}|^2}{|z^{\alpha^1}|^{2c}+\dots +|z^{\alpha^k}|^{2c}}$ is locally integrable near the origin.

The denominator $|z^{\alpha^1}|^{2c}+\dots +|z^{\alpha^k}|^{2c}$ does not vanish at a point off the coordinate axes (i.e. if all the coordinates are non-zero), so local integrability near these points follows immediately. For the other zeros of $|z^{\alpha^1}|^{2c}+\dots +|z^{\alpha^k}|^{2c}$, the problem is reduced to integrability near the origin in a smaller dimension, then handled as above.  Indeed, suppose $|z^{\alpha^1}|^{2c}+\dots +|z^{\alpha^k}|^{2c}$ vanishes at $w=(w_1,\dots ,w_n)\neq 0$. Let $\widetilde{w}=(w_{i_1},..,w_{i_m})$ be the zero coordinates of $w$ and let $\widetilde{\beta}=(\beta_{i_1},\dots ,\beta_{i_m})\text{ and } \widetilde{\alpha}^j=(\alpha^j_{i_1},\dots ,\alpha^j_{i_m})$ denote the corresponding multi-indices. In a small neighborhood of $w$, the pair of estimates
\begin{align*}
\frac{|\widetilde{z}^{\widetilde{\beta}}|^2}{|\widetilde{z}^{\widetilde{\alpha}^1}|^{2c}+\dots +|\widetilde{z}^{\widetilde{\alpha}^k}|^{2c}}\lesssim \frac{|z^{\beta}|^2}{|z^{\alpha^1}|^{2c}+\dots +|z^{\alpha^k}|^{2c}}\lesssim\frac{|\widetilde{z}^{\widetilde{\beta}}|^2}{|\widetilde{z}^{\widetilde{\alpha}^1}|^{2c}+\dots +|\widetilde{z}^{\widetilde{\alpha}^k}|^{2c}},
\end{align*}
hold, since the non-zero coordinates and their powers are bounded from below and above. Thus, the integrability of the middle rational expression
near $w$ in $\mathbb{C}^n$ is equivalent to the integrability of the bounding rational expression
near the origin in $\mathbb{C}^m$ for some $m\leq n$. Furthermore, $t_1c\alpha^1+\dots +t_kc\alpha^k\prec \beta+\one$ implies $t_1c\widetilde{\alpha}^1+\dots +t_kc\widetilde{\alpha}^k\prec \widetilde{\beta}+\one$, so the integrability of the bounding rational expression follows from the same argument that showed \eqref{ratio} was integrable near the origin.

This proves one implication in (ii): $\beta+\one \in P^{\circ}(c\fa)\implies z^{\beta}\in \JfcaAN$.
\vskip 0.5cm

For the converse, first observe that $P(c\fa)$ is cut-out by finitely many hyperplanes $H_1,\dots ,H_N$. Each hyperplane has an equation of the form
$$H_i: v^i_1x_1+\dots +v^i_nx_n=\kappa^i \hskip 1cm \text{ where } \kappa^i = 0\text{ or } 1,$$
for a vector $(v^i_1,\dots ,v^i_n)$ normal to the hyperplane. For simplicity, write these equations as
$$\left<\mathbf{v}^i,\mathbf{x}\right>=\kappa^i,\qquad i=1,\dots , N.$$
Since $\mathbf{x}+\mathbb{R}^n_+\subset P(\fa)$ for any point $\mathbf{x}\in P(\fa)$, the normal vector $\mathbf{v}^i$ to any $H_i$ can be taken to point toward the inside of the solid Newton polyhedra, i.e., all the components of each vector $\mathbf{v}^i$ may be assumed non-negative. Each $H_i$ splits $\mathbb{R}^n$ into two half-spaces. With this choice of sign, $P(c\fa)$ lies in the region where $\left<\mathbf{v}^i,\mathbf{x}\right>\geq \kappa^i$ and
\begin{equation*}\label{hyperplanes}
\gamma \in P(c\fa) \text{ if and only if }\left<\mathbf{v}^i,\gamma\right> \geq \kappa^i\text{ for all }i=1,\dots ,N.
\end{equation*}

Now consider a multi-index $\beta$ such that $\beta+\one \not \in P^{\circ}(c\fa)$. Since $\beta+\one \not \in P^{\circ}(c\fa)$, there exists a hyperplane $H^\beta\in\left\{H_1,\dots, H_N\right\}$, with normal vector $\mathbf{v}^\beta$, such that $\left< \mathbf{v}^\beta, \beta+\one\right>\leq \left< \mathbf{v}^\beta, c\alpha^j\right>$ for all $j=1,\dots , k$. At least one of the components of $\mathbf{v}^\beta$ is non-zero, which can be taken to be the first component without loss of generality. Scale $\mathbf{v}^\beta$ to obtain a vector $\mathbf{b}=\left(1,b_2,\dots ,b_n\right)$, $b_i\geq0$ such that
\begin{equation}\label{necessary}
\left<\beta+\one,\mathbf{b}\right> \leq\min_{1\leq i\leq k}\left<c\alpha^i,\mathbf{b}\right>:=m.
\end{equation}

It is now straightforward to show that the monomial $z^{\beta}$ fails to be locally integrable near the origin. If $U$ is the unit polydisc centered at the origin in $\mathbb{C}^n$,

\begin{align*}
\int_U &\frac{|z^{\beta}|^2}{|z^{\alpha^1}|^{2c}+\dots+|z^{\alpha^k}|^{2c}}\, dV(z)\\ &\geq\int_0^1\int_{\frac 12 r_1^{b_2}}^{r_1^{b_2}}\dots\int_{\frac 12 r_1^{b_n}}^{r_1^{b_n}}~\frac{r_1^{2\beta_1+1}\dots r_n^{2\beta_n+1}\,\, dr_1\dots dr_n}{r_1^{2c\alpha^1_1}\dots r_n^{2c\alpha^1_n}+\dots +r_1^{2c\alpha^k_1}\dots r_n^{2c\alpha^k_n}}\\
&\geq \int_0^1\int_{\frac 12 r_1^{b_2}}^{r_1^{b_2}}\dots\int_{\frac 12 r_1^{b_n}}^{r_1^{b_n}}~\frac{r_1^{2\beta_1+1}\dots r_n^{2\beta_n+1}\,\, dr_1\dots dr_n}{r_1^{2\alpha^1_1}\dots r_1^{b_n2\alpha^1_n}+\dots+r_1^{2\alpha^k_1}\dots r_1^{b_n2\alpha^k_n}}\\
&= \text{const. } \int_0^1\frac{r_1^{2\left<\beta+\one,\mathbf{b}\right>}}{r_1\left(r_1^{2\left<c\alpha^1,\mathbf{b}\right>}+\dots +r_1^{2\left<c\alpha^k,\mathbf{b}\right>}\right)}\,
dr_1\\
&\gtrsim\int_0^1\frac{r_1^{2\left<\beta+\one,\mathbf{b}\right>}}{r_1r_1^{2m}}\, dr_1=\infty.
\end{align*}
Thus,  $z^{\beta} \not \in \Jfca$, which proves the second implication in (ii): $z^{\beta}\in \Jfca\implies\beta+\one\in P^{\circ}(c\fa)$.

To complete the proof of Theorem \ref{main}, it remains to show (i) holds. Observe that the denominator in Definition \ref{an_cmultiplier} is radial if $\fa$ is monomial. Suppose $g(z)=\sum_{\gamma}g_{\gamma}z^{\gamma}$ is an element of the ideal $\JfcaAN$, i.e.,
$$\int_U \frac{|g(z)|^2}{|z^{\alpha^1}|^{2c}+\dots +|z^{\alpha^k}|^{2c}}\, dV(z)<\infty.$$
Using the orthogonality of the monomials $z^\nu$, $\nu\in \left(\mathbb{N}\cup \left\{0\right\}\right)^n$, in $L^2(U)$, this finite integral can be rewritten as follows:
\begin{align*}
\int_U \frac{|g(z)|^2}{|z^{\alpha^1}|^{2c}+\dots +|z^{\alpha^k}|^{2c}}\, dV(z)&=\int_U\sum_{j=1}^k \frac{1}{|z^{\alpha^j}|^{2c}}\sum_{\gamma,\eta}g_{\gamma}\overline{g_{\eta}}z^{\gamma}\overline{z^{\eta}}\, dV(z)\\
&=\sum_{\gamma}|g_{\gamma}|^2\int_U\sum_{j=1}^k \frac{|z^{\gamma}|^2}{|z^{\alpha^j}|^{2c}}\, dV(z).
\end{align*}
Since there is no cancellation amongst the terms in the last expression, each $\int_U\sum_{j=1}^k \frac{|z^{\gamma}|^2}{|z^{\alpha^j}|^{2c}}$ must be finite.
This says that each monomial $z^{\gamma}$ appearing in $g$ must belong to $\JfcaAN$. Consequently, $\JfcaAN$ is generated by monomials, as claimed by statement (i).\\
\end{proof}

\section{Integral Closure of Ideals}

If $\fa=\left(p_1(z),\dots ,p_k(z)\right)\subset \mathbb{C}[z_1,\dots ,z_n]$ is a non-zero ideal,  the \textit{integral closure} of $\fa$ is
another geometrically natural ideal associated to $\fa$. As in the case of multiplier ideals, there are two ways to define the integral closure, algebraically and analytically, and the two definitions look dissimilar.

The algebraic definition of integral closure is
\begin{definition}\label{closure_al}
Let $\fa\subset \mathbb{C}[z_1,\dots ,z_n]$ be a non-zero ideal. A polynomial
$f(z)\in \mathbb{C}[z_1,\dots ,z_n]$ belongs to $\closedfaAL$ if there exist $b^j\in \fa^j$ such that $f$ satisfies the following equation
\begin{equation}\label{closure_equation}
f^N+b_1f^{N-1}+\dots +b_{N-1}f+b_N=0.
\end{equation}
Here $\fa^j$ is the j-th power ideal of $\fa$.
\end{definition}
\noindent For equivalent algebraic definitions of integral closure, its properties and connection to other ideas, see \cite{Lazarsfeldbook} pages 216--221, and the relevant sections in \cite{eisenbudbook}.

The analytic definition of integral closure is
\begin{definition}\label{closure_an}
Let $\fa\subset \mathbb{C}[z_1,\dots ,z_n]$ be a non-zero ideal. A polynomial
$f(z)\in \mathbb{C}[z_1,\dots ,z_n]$ belongs to $\closedfaAN$ if for any $q\in \mathbb{C}^n$ there exists a neighborhood $\mathcal{V}$ of $q$ and $C>0$ such that
\begin{equation}\label{closure_ineq}
\left|f(z)\right|\leq C\sum_{i=1}^k\left|p_i(z)\right| \text{ for all }z\in \mathcal{V},
\end{equation}
where $\fa =(p_1\dots, p_k).$
\end{definition}
\noindent This definition of integral closure is appealing because \eqref{closure_ineq} can be used to test membership in $\closedfaAN$ transparently.

Showing that Definition \ref{closure_al} is equivalent to Definition \ref{closure_an} is not trivial. However Teissier \cite{teissier82} and Lejeune-Jalabert and Teissier \cite{Teissier2008} showed that $\closedfaAL$ and
$\closedfaAN$ do in fact coincide, for any ideal $\fa$ generated by polynomials in $\mathbb{C}[z_1,\dots ,z_n]$. An essential ingredient in
Teissier's proof is taking a log resolution of the ideal $\fa$.

If $\fa$ is generated by arbitrary polynomials, resolving $\fa$ in some fashion may perhaps be necessary in order to obtain a result of this type. But if $\fa$ is a monomial ideal, we give a simple proof of Teissier's theorem below (Theorem \ref{definition}) that does not use a resolution of singularities theorem. The first
step of this proof is to describe $\closedfaAN$ in terms of $P(\fa)$. The proof of this description uses the Arithmetic Mean-Geometric Mean inequality \eqref{AGinequality} in much the same way as was done in Section 3.

\begin{theorem}\label{description} Let $\fa=\left(z^{\alpha^1},\dots ,z^{\alpha^k}\right)$ be a monomial ideal then
\begin{enumerate}
\item[(i)] $\closedfaAN$ is also a monomial ideal.
\item[(ii)] $z^{\beta} \in \closedfaAN \text{ if and only if } \beta \in P(\fa).$
\end{enumerate}
\end{theorem}

\begin{proof} As in the proof of Theorem \ref{main}, the equivalence in (ii) will be shown first. Let $\beta$ be a multi-index such that $\beta \in P(\fa)$; then there exist $t_i\geq0$ satisfying $t_1+\dots +t_k=1$ such that $t_1\alpha^1+\dots +t_k\alpha^k\preceq \beta$. In a neighborhood of the origin, the inequality
\begin{align*}
|z^{\beta}|&=|z_1|^{\beta_1}\dots |z_n|^{\beta_n}\\&\lesssim |z_1|^{t_1\alpha^1_1+\dots +t_k\alpha^k_1}\dots |z_n|^{t_1\alpha^1_n+\dots +t_k\alpha^k_n} = A
\end{align*}
trivially holds. However rearranging these terms and applying \eqref{AGinequality} gives
\begin{align*}
A&=|z_1|^{t_1\alpha^1_1}|z_2|^{t_1\alpha^1_2}\dots |z_n|^{t_1\alpha^1_n}\dots |z_1|^{t_k\alpha^k_1}|z_2|^{t_k\alpha^k_2}\dots |z_n|^{t_k\alpha^k_n}\\&=\left|z^{\alpha^1}\right|^{t_1}\dots\left|z^{\alpha^k}\right|^{t_k}
\lesssim \left|z^{\alpha^1}\right|+\dots+\left|z^{\alpha^k}\right|.
\end{align*}
This shows that $f(z)=z^{\beta}$ satisfies \eqref{closure_ineq} for $z$ near the origin. Near other points, $q$, inequality \eqref{closure_ineq} for $f(z)=z^{\beta}$ either holds trivially, or holds by the above argument, after eliminating the non-zero coordinates of $q$ as done in the proof of Theorem \ref{main}. Consequently, $\beta\in P(\fa)\implies z^{\beta}\in \closedfaAN$.

For the converse, take $\beta \not \in P(\fa)$ and observe, as in the proof of Theorem \ref{main}, that there exists a vector $\mathbf{b}=(1,b_2,\dots ,b_n)$ such that
\begin{equation}\label{necessary2}
\left<\beta,\mathbf{b}\right> \hskip .3cm < \min_{1\leq i\leq k}\left<\alpha^i,\mathbf{b}\right>:=m.
\end{equation}
[Note the strict inequality in \eqref{necessary2}, in contrast to \eqref{necessary}.]
Consider the rectangle $[0,1]\times \left[\frac{r_1^{b_2}}{2},r_1^{b_2}\right]\times\dots \times\left[\frac{r_1^{b_n}}{2},r_1^{b_n}\right]\subset \mathbb{R}^n_+$. For $z$ in this rectangle, we have
$$
\left|z^{\beta}\right|\sim r_1^{\left<\beta,\mathbf{b}\right>}\quad\text{ while }\quad
\sum_{j=1}^k\left|z^{\alpha^j}\right|\sim r_1^m.
$$
However, taking $r_1\to 0$ and noting \eqref{necessary2}, it follows that inequality \eqref{closure_ineq} fails for $f(z)=z^{\beta}$ near the origin.
Thus, $\beta \not \in P(\fa)\implies z^{\beta}\not \in \closedfaAN$, which completes the proof of (ii).

It remains to show that (i) holds. Let $g(z)=\sum_{\gamma}g_{\gamma}z^{\gamma}$ be an arbitrary polynomial in $\closedfaAN$. We claim that each exponent $\gamma$
appearing in this expression must reside in $P(\fa)$; this implies (by the already established (ii)) that each monomial $z^{\gamma}$ in $g(z)$ belongs to $\closedfaAN$, which implies that $\closedfaAN$ is a monomial ideal.

Suppose there exists an exponent $\gamma_0\not \in P(\fa)$. By \eqref{necessary2}, there exists $\mathbf{b}=(1,b_2,\dots ,b_n)$, $b_i\geq0$ such that
\begin{equation}\label{t_0m_0}
t_0=:\left<\gamma_0,\mathbf{b}\right> < \min_{1\leq i \leq k}\left<\alpha^k,\mathbf{b}\right>:=m_0.
\end{equation}
The basic idea is to approach the origin along the curve $\left(r,r^{b_2},\dots ,r^{b_n}\right)$ and show that inequality \eqref{closure_ineq} fails. This idea is slightly complicated by the fact that
$$g\left(r,r^{b_2},\dots r^{b_n}\right)=\sum_{\gamma}g_{\gamma}\left(r,r^{b_2},\dots r^{b_n}\right)^{\gamma}=\sum_{\gamma}g_{\gamma}r^{<\gamma,\mathbf{b}>}$$
might not contain the term $r^{t_0}$, due to cancellation.

This difficulty is easily overcome. Since none of the coefficients $g_{\gamma}$ are zero,
$$w\to \mathcal{P}(w)=\sum^{\ast}_{\gamma}g_{\gamma}\left(w_1r,w_2r^{b_2},\dots ,w_nr^{b_n}\right)^{\gamma}=
r^{t_0}\sum^{\ast}_{\gamma}g_{\gamma}w^{\gamma}$$
is a non-trivial polynomial of $w$, for any fixed $r$, where $\sum^{\ast}$ denotes the summation over the multi-indices $\gamma$ such that $\left<\gamma,\mathbf{b}\right>=t_0$. This  polynomial $\mathcal{P}\in \mathbb{C}\left[w_1,\dots,w_n\right]$ could vanish at $w=(1,\dots,1)$, but at a generic point in $\mathbb{C}^n$ it is non-vanishing. Simply take any such point $\tilde w=\left(\tilde w_1,\dots , \tilde w_n\right)$ with $\left|\tilde w_j\right|=1$ for $j=1,\dots , n$ and approach
the origin along the curve $\left(\tilde w_1r,\tilde w_2r^{b_2},\dots,\tilde w_nr^{b_n}\right)$.

As there are only finitely many exponents $\gamma$ in the polynomial $g(z)$, it is easy to compute $\min_{\gamma}\left<\gamma,\mathbf{b}\right>$ and see that
$\min_{\gamma}\left<\gamma,\mathbf{b}\right>\leq t_0$ (with the sign choice mentioned above \eqref{necessary}. Thus, as $r\to 0$,
$$r^{t_0}\lesssim\left|g\left(w_1r,w_2r^{b_2},\dots,w_nr^{b_n}\right)\right|.$$
On the other hand, as $r\to 0$ on the curve $\left(w_1r,w_2r^{b_2},\dots,w_nr^{b_n}\right)$, 
 $$\left|z^{\alpha^1}\right|+\dots +\left|z^{\alpha^k}\right|=|w|^{\alpha^1}r^{<\alpha^1,\mathbf{b}>}+\dots +|w|^{\alpha^k}r^{<\alpha^k,\mathbf{b}>}\sim r^{m_0},$$
 since each $w_i\not = 0$.

Since \eqref{t_0m_0} holds, these two inequalities show \eqref{closure_ineq} fails for $g$. This contradicts the assumption that $g\in\closedfaAN$ and finishes the proof of (i).
\end{proof}

With Theorem \ref{description} in hand, a proof of Teissier's general result on $\closedfaAN=\closedfaAL$ can be given for the case that $\fa$ is a monomial ideal. The interesting implication is that inequality \eqref{closure_ineq} forces an algebraic relationship between $f$ and the generators of $\fa$, i.e., that (a) implies (b) in the following result.

\begin{theorem}\label{definition}{\rm (Teissier)}
Let $\fa=\left(z^{\alpha^1},\dots ,z^{\alpha^k}\right)$ be a monomial ideal in $\mathbb{C}[z_1,\dots ,z_n]$.
Given $f\in\mathbb{C}[z_1,\dots ,z_n]$, the following are equivalent:
\begin{enumerate}
\item[(a)] $f\in \closedfaAN$.
\item[(b)] $f\in \closedfaAL$.
\item[(c)] There exists a non-zero ideal $\fb\subset \mathbb{C}[z_1,\dots ,z_n]$ such that $f\cdot\fb\subset\fa\cdot\fb.$
\end{enumerate}
\end{theorem}

\begin{proof} The implication $(c)\implies(b)$ follows from the \textit{determinant trick}, which is well-explained in \cite{Lazarsfeldbook} pages 217--218.

Next, consider the statement $(b)\implies (a)$. This implication holds in general, i.e., without assuming that $\fa$ is monomial, so we shall write $\fa =\left( p_1,\dots p_k\right)$ for notational convenience. Assume that relation \eqref{closure_equation} holds. Fix a point $q\in \mathbb{C}^n$ and a neighborhood $\mathcal{V}$ of $q$. Since the coefficients $b_i\in a^i$ appearing in  \eqref{closure_equation} have the form
$b_i(w)=\sum_{j_1,\dots ,j_i}c_{j_1,..,j_i}p_{j_1}(w)\dots p_{j_i}(w)$
for some $c_{j_1,\dots ,j_i}\in \mathbb{C}$, simple estimation shows that there exists $C>0$, depending on $q, \mathcal{V}$ and $b^i$, such that for all $w\in \mathcal{V}$
\begin{align}\label{estimate}
\left|b_i(w)\right|=\left|\sum_{j_1,\dots ,j_i}c_{j_1,..,j_i}p_{j_1}(w)\dots p_{j_i}(w)\right|\leq C\left(|p_1(w)|+\dots +|p_k(w)|\right)^i.
\end{align}
Now re-write \eqref{closure_equation} and use \eqref{estimate}:
\begin{align*}
\left|f^N(w)\right|&=\left|b_1(w)f^{N-1}(w)+\dots +b_{N-1}(w)f(w)+b_N(w)\right|\\
&\leq C\left\{ \left(|p_1(w)|+\dots +|p_k(w)|\right)\left|f^{N-1}(w)\right|+\dots \right. \\
&+\left.\left(|p_1(w)|+\dots +|p_k(w)|\right)^{N-1}|f(w)|+\left(|p_1(w)|+\dots +|p_k(w)|\right)^N\right\}\\
&\leq C\left( \sum_{j=1}^k|p_j(w)|+|f(w)|\right)^N-C|f(w)|^N.
\end{align*}
This yields the estimate
\begin{equation*}
(C+1)|f(w)|^N\leq C\left(\sum_{j=1}^k|p_j(w)|+|f(w)|\right)^N,
\end{equation*}
or, after taking roots,
\begin{equation*}
|f(w)|\leq \left(\frac{C}{C+1}\right)^{1/N}\left(\sum_{j=1}^k|p_j(w)|+|f(w)|\right).
\end{equation*}
Since the constant on the right-hand side of this inequality is $< 1$, the $|f(w)|$ term can be absorbed, yielding the estimate 
$$|f(w)|\leq C'\sum_{j=1}^k|p_j(w)|.$$
for all $w\in \mathcal{V}$. This is exactly \eqref{closure_ineq}, which finishes the implication $(b)\implies (a)$.

It remains to prove $(a)\implies(c).$ For this we use Theorem \ref{description} to establish the following Lemma.

\begin{lemma}\label{power_lemma}
There exists an integer $A>0$ such that
\begin{equation}\label{Apower}
\closedfaA=\fa\cdot\closedfaAminus
\end{equation}
\end{lemma}
\begin{proof}
The ideal $\faA$ is also a monomial ideal and is generated by monomials $(z^{\gamma^i})_{i=1,\dots ,M}$ where $\gamma^i$ is of the form $c^i_1\alpha^1+\dots +c^i_k\alpha^k$ for some non-negative integers $c^i_j$ such that $c^i_1+\dots +c^i_k=A$ for any $i=1,\dots ,M$.

Let $z^{\beta}\in\closedfaA$.
By Theorem \ref{description}, $\beta\in P\left(\faA\right)$, so there exist real numbers $t_1,\dots ,t_M\geq 0$ summing to 1 such that
$$t_1\gamma^1+\dots +t_M\gamma^M\preceq \beta.$$
An expansion of the left-hand side of this line yields
\begin{align*}
t_1\gamma^1+\dots +t_M\gamma^M&=t_1\left(c^1_1\alpha^1+\dots +c^1_k\alpha^k)+\dots +t_M(c^M_1\alpha^1+\dots +c^M_k\alpha^k\right)\\
&=D_1\alpha^1+\dots +D_k\alpha^k
\end{align*}
where $D_i=t_1c^1_i+\dots +t_Mc^M_i$. Note that
\begin{align*}
D_1+\dots +D_k&=t_1c^1_1+\dots +t_Mc^M_1+\\
&+\vdots~\vdots~\vdots\\
&+t_1c^1_k+\dots +t_Mc^M_k\\
&=t_1A+\dots +t_MA=A
\end{align*}
Consequently, if $A\geq k$ then one of $D_i$ must be greater than 1. Let's say $D_1\geq 1$. We now claim that
\begin{equation*}
z^{(D_1-1)\alpha^1+\dots +D_k\alpha^k}\in \closedfaAminus.
\end{equation*}

The ideal $\faAminus$ is a monomial ideal and is generated by monomials $(z^{\eta^i})_{i=1,\dots ,L}$ where $\eta^i$ is of the form $d^i_1\alpha^1+\dots +d^i_k\alpha^k$ for some non-negative integers $d^i_j$ such that $d^i_1+\dots +d^i_k=A-1$ for any $i=1,\dots ,L$. By Theorem \ref{description}, $z^{(D_1-1)\alpha^1+\dots +D_k\alpha^k}\in \closedfaAminus$ if and only if there exists real numbers $s_1,\dots ,s_L\geq0$ summing to 1 such that
\begin{equation}\label{test}
s_1\eta^1+\dots +s_L\eta^L\preceq(D_1-1)\alpha^1+\dots +D_k\alpha^k.
\end{equation}

To see \eqref{test}, re-express the $(\eta^i)$ as $\eta^1=(A-1)\alpha^1, \eta^2=(A-1)\alpha^2,\dots ,\eta^k=(A-1)\alpha^k$. The choice for the coefficients $s_i$'s then presents itself:
\begin{align*}
s_1&=\frac{D_1-1}{A-1}\\
s_2&=\frac{D_2}{A-1}\\
&\vdots\\
s_k&=\frac{D_k}{A-1}\\
s_j&=0\text{ for }j>k.
\end{align*}
This choice guarantees \eqref{test}, which implies $z^{(D_1-1)\alpha^1+\dots+D_k\alpha^k}\in \closedfaAminus$. Thus, $z^{\beta}\in \fa\cdot\closedfaAminus$ and one half of the lemma
\begin{equation*}
\closedfaA\subset\fa\cdot\closedfaAminus
\end{equation*}
is proved.

The opposite inclusion is easier and holds in a more general setting. For any two ideals, it is easy to check by using Definition \ref{closure_an} that
$$\fa\cdot\fb\subset\closedfaAN\cdot\closedfbAN\subset\overline{\fa\cdot\fb}^{an}.$$
This simple observation finishes the proof of the lemma.
\end{proof}
The proof of this lemma also gives an effective number $A$ in the equation \eqref{Apower}. The exponent $A$ can be taken equal to the number of generators of the monomial ideal $\fa$.\\

The proof of $(a)\implies(c)$ is now finished by taking $\fb=\closedfaAminus$. Lemma \ref{power_lemma} implies
\begin{equation*}
f\cdot\fb=f\cdot\closedfaAminus\subset\closedfa\cdot\closedfaAminus\subset\closedfaA=\fa\cdot\closedfaAminus=\fa\cdot\fb.
\end{equation*}
This is exactly $(c)$, and concludes the proof of Theorem \ref{definition}.\\
\end{proof}

\bibliographystyle{plain}
\bibliography{multiplier_ideals}
\end{document}